\documentclass[12pt, reqno]{amsart}
\usepackage[english]{babel}
\usepackage{amsmath, amsthm, amscd, amsfonts, amssymb, graphicx, color,latexsym}
\usepackage[bookmarksnumbered, colorlinks, plainpages]{hyperref}
\usepackage{booktabs}
\usepackage[utf8]{inputenc}
\usepackage{amsthm}
\usepackage{amsmath}
\usepackage{pgf,tikz}
\usepackage{amsfonts,amssymb,enumerate}
\newtheorem{Def}{Definition}[subsection]
\newtheorem{theo}[Def]{Theorem}
\newtheorem{prop}[Def]{Proposition}
\newtheorem{lemma}[Def]{Lemma}

\newtheorem{introthm}{Theorem}

\newtheorem{introcor}[introthm]{Corollary}

\theoremstyle{definition}
\newtheorem{rk}{Remark}[subsection]

\newtheorem{quest}{Question}

\DeclareMathOperator\Span{Span}

\DeclareMathOperator\Ima{Im}
\DeclareMathOperator\End{End}
\DeclareMathOperator\Isom{Isom}
\newcommand{\R}{\mathbb{R}}

\newcommand{\C}{\mathbb{C}}

\newcommand{\Hq}{\mathbb{H}}
\newcommand{\p}{\mathbb{P}}
\newcommand{\Z}{\mathbb{Z}}

\newcommand{\F}{\mathcal{F}}
\newcommand{\g}{\mathfrak{g}}

\newcommand{\tor}{\mathfrak{t}}

\newcommand{\cH}{\mathcal{H}}
\newcommand{\cV}{\mathcal{V}}
\newcommand{\cD}{\mathcal{D}}

\newcommand{\su}{\mathfrak{su}}

\newcommand{\lied}{\mathcal{L}}

\everymath{\displaystyle}
\usepackage[backend=biber, style=alphabetic, sorting=nyt, maxalphanames=4]{biblatex}
\begin{document}
\usetikzlibrary{positioning}
\setcounter{page}{1}

\title[$SU(3)$-structures on quotients of $3$-Sasakian orbifolds]{$SU(3)$-structures on quotients of $3$-Sasakian orbifolds}

\author{Quentin Peres, Dimitrios Tsimpis}
\address{Quentin Peres \\Institut Camille Jordan, Université Claude Bernard Lyon 1 \\ 21 Av. Claude Bernard, 69100 Villeurbanne, France}
\email{peres@math.univ-lyon1.fr}
\address{Dimitrios Tsimpis\\Institut de Physique des Deux Infinis de Lyon, Universit\'e de Lyon, UCBL, UMR 5822, CNRS/IN2P3, 4 rue Enrico Fermi, 69622 Villeurbanne Cedex, France}
\email{d.tsimpis@ipnl.in2p3.fr}
\begin{abstract}
We show that if $(S,g,\xi_i,\eta_i,\Phi_i)$ is a quasi-regular $3$-Sasakian orbifold of dimension $7$, then its quotient $Z$ by one of the Reeb vector fields inherits an $SU(3)$-structure parametrized by two real parameters. We compute its torsion and show that it is an LT-structure. Moreover, for certain values of the parameters, we can provide a nearly Kähler structure on $Z$ appearing as a special case of our construction. We give several examples in the regular,  quasi-regular and orbifold cases. 
\end{abstract} \maketitle
\tableofcontents

\section{Introduction}
In string theory and supergravity,  $SU(3)$-structures on six-dimensional manifolds are of particular importance, notably in the context of flux compactifications, where the intrinsic torsion is closely related to the background fluxes and encodes relevant geometric and physical data \cite{Gra_a_2006}. Calabi–Yau threefolds correspond to the distinguished case of torsion-free $SU(3)$-structures, arising in the absence of fluxes in the simplest compactification settings.  \\ 
\indent An \textit{$SU(n)$-structure} on a smooth manifold $M$ of real dimension $2n$ is given by an almost complex structure $J$, a non-degenerate real $(1,1)$-form $\omega$ such that the symmetric tensor $g=\omega(\cdot,J\cdot)$ defines a Riemannian metric on $M$ and a complex $(n,0)$ volume form $\Omega$ which is locally decomposable. This can be summed up by the following algebraic equations:
$$\Omega\wedge\omega=0,$$ $$\Omega\wedge\overline{\Omega}=c_n\frac{\omega^n}{n!}$$ where $c_n=2^n(-i)^{n^2}$. The algebraic equations defining the $SU(n)$-structure on $M$ completely determine the almost complex structure $J$ and the Riemannian metric $g$. Alternatively, an $SU(n)$-structure on $M$ is an \textit{almost hermitian structure} $(J,\omega,g)$ such that the first Chern class of $M$ vanishes. Given an $SU(n)$-structure on $M$ determined by $\omega$ and $\Omega$, one can compute its \textit{torsion} by calculating the exterior derivatives of $\omega$ and $\Omega$:
$$d\omega=\frac{3}{2^{n-2}}i^{(n+1)^2}\Ima(\overline{W_1}\cdot\Omega)+W_4\wedge\omega+W_3,$$ $$d\Omega=W_1\wedge\omega^2+W_2\wedge\omega+\overline{W_5}\wedge\Omega,$$ where $W_1$ is an $(n-3,0)$-form, $W_2$ a primitive $(n-2,1)$-form, $W_3$ a real primitive $(2,1)+(1,2)$-form, $W_4$ a real $(1,0)+(0,1)$-form and $W_5$ is a $(1,0)$-form. The $W_i$ are called the \textit{torsion classes} and characterize a geometric property of the structure and were first described in \cite{Grayhervella}. In this paper, we shall be interested in a particular type of $SU(n)$-structure, called LT-structure, see Definition \ref{LT}.\\ 
\indent In \cite{Larfors_2010}, Larfors, Lüst and Tsimpis provide a method to construct $SU(3)$-structures on six-dimensional toric manifolds. This was then generalized in \cite{peres2026sunstructuresquotienttorusactions} for arbitrary dimensions. The method is based on reversing one direction of the almost complex structure by using a special $1$-form called a \textit{twist form}, see Definition \ref{twistdef}. A similar approach was used by Eells and Salamon in \cite{EellsSalamon} where they use the Levi-Civita connection to reverse one direction of the integrable complex structure of any twistor space over a four oriented Riemannian manifold.  Our goal is to apply these different results in the context of $3$-Sasakian manifolds,  which were first introduced in \cite{Kuo1970OnAC}, and also generalize them to orbifolds.
\subsection{$3$-Sasakian manifolds} A Riemannian manifold $(S,g)$ has a \textit{$3$-Sasakian structure} if its Riemannian cone $(C(S)=S\times\R_{>0},\bar g=dr^2+r^2g)$ is \textit{hyperkähler}. In particular, $S$ is of dimension $4n+3$ for some positive integer $n$ and has three extra-structures induced by the hyperkähler structure of the cone. More precisely, let $J_i$ for $i=1,2,3$ be the three integrable complex structures compatible with the metric $\bar g$ satisfying the quaternionic relations $$J_iJ_j=-\delta_{ij}+\epsilon_{ijk}J_k$$ where $\epsilon_{ijk}$ is the Levi-Civita symbol. These complex structures define three nowhere-vanishing vector fields $\xi_i=J_i(r\partial_r)$, called the Reeb vector fields that can be restricted to $S$ which is isometrically identified with the link $\{r=1\}$ in $C(S)$. The $1$-forms dual to the Reeb vector fields that we denote $\eta_i$ define three contact structures on $S$. If $\nabla$ denotes the Levi-Civita connection of $(S,g)$, then the tensor $\Phi_i=\nabla\xi_i$ define three CR-structures on $S$. Proposition \ref{trisasakien} shows that the three Reeb vector fields satistfy the $\su(2)$ Lie algebra relations. With this observation, the quotient $B$ of $S$ by the foliation generated by the three Reeb vector fields provide a \textit{quaternionic Kähler orbifold} structure in the quasi-regular case. In addition, one can also consider the quotient of $S$ by the foliation generated by one of the Reeb vector fields, this quotient $Z$ will be called the \textit{twistor space} of $S$ and has an orbifold structure in the quasi-regular case. Several properties of these different quotients will be studied in Section \ref{3.2} and summed up in the following diagram:
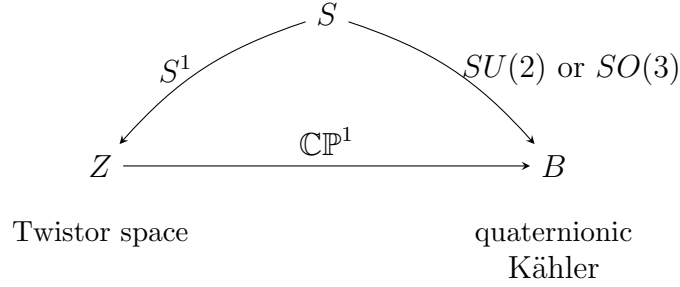
\[
\begin{tikzpicture}[>=stealth]

\node (S) at (0,2) {$S$};
\node (Z) at (-3,0) {$Z$};
\node (B) at (3,0) {$B$};

\draw[->] (S) to[bend right=15] node[left] {$S^1$} (Z);
\draw[->] (S) to[bend left=15] node[right] {$SU(2)\ \text{or}\ SO(3)$} (B);

\draw[->] (Z) -- node[above] {$\mathbb{CP}^1$} (B);

\node[below=3mm of Z] {\small Twistor space};
\node[below=3mm of B, align=center] {\small quaternionic\\K\"ahler};

\end{tikzpicture}
\]
\subsection{Main results} The framework of $3$-Sasakian manifolds provides a natural application of the results of \cite{peres2026sunstructuresquotienttorusactions}. More precisely, let $(S,g,\xi_i,\eta_i,\Phi_i)$ be a quasi-regular or regular $3$-Sasakian manifold of dimension $7$ and consider its hyperkähler cone $(C(S),\bar g, J_i)$. The cone being hyperkähler implies that it also has an explicit Calabi-Yau $4$-fold structure. Indeed, let $\omega_i$ be the Kähler forms associated to $\bar g$ and $J_i$ on $C(S)$. Fixing the Kähler structure $(\omega_3,J_3)$, the form $\Omega_{(4)}=\frac{1}{2}(\omega_1+i\omega_2)^2$ is a holomorphic $(4,0)$-form. If $S$ is further assumed to be compact, the Reeb vector field $\xi_3$ generates a locally free $S^1$ action on $C(S)$ which is hamiltonian. Using the other two Sasakian structures on $S$, we construct a twist form, see Definition \ref{twistdef}, and prove the following result:
\begin{introthm}\label{theo1}
    Let $(S,g,\xi_i,\eta_i,\Phi_i)$ be a (quasi)-regular $7$ dimensional $3$-Sasakian manifold and $Z$ the quotient of $S$ by the foliation generated by $\xi_3$. Denote $(C(S),\bar g,J_i)$ the hyperkähler cone of $S$ and the associated Kähler forms $\omega_i:=\bar g(J_i\cdot,\cdot)$. The holomorphic volume form $2\Omega_{(4)}:=(\omega_1+i\omega_2)^2$ is of charge $4i$ and $\eta:=r^2(\eta_1+i\eta_2)$ is a twist form of charge $2i$. Then, the following forms: $$\omega:=\omega_3-\frac{i}{|\eta|^2}\eta\wedge\bar\eta,\,\, \Omega:=-i\bar\eta\wedge(\bar\eta\cdot(\iota_{\xi_3}\Omega_{(4)}))$$ produce an LT structure on $Z$. Moreover, the non-vanishing torsion classes are
    $$W_1=\frac{8}{3},\, W_2=\frac{W_1}{2}(\omega+\frac{3i}{2}\eta\wedge\bar\eta).$$
\end{introthm}
Alternatively, one can also work directly on the horizontal distribution generated by $\xi_3$ to produce a parametric family of $SU(3)$-structures on the quotient $Z$.
\begin{introthm}\label{theo2}
    Let $(S,g,\xi_i,\eta_i,\Phi_i)$ be a (quasi)-regular $7$ dimensional $3$-Sasakian orbifold. Denote $Z$ the quotient of $S$ by the foliation generated by $\xi_3$. Consider the $2$-forms $\alpha_i$ defined by $$2\alpha_i:=d\eta_i-\epsilon_{ijk}\eta_j\wedge\eta_k.$$ For $a,b\in\R$ and $c\in\C$, we introduce $$\omega_{a,b}:=a\alpha_3-\frac{i}{2}b^2\eta\wedge\bar\eta,\, \Omega_c:=c\alpha\wedge\bar\eta.$$ These two forms are basic with respect to $\xi_3$ and define an $SU(3)$-structure on $Z$ iff $$a>0,\,\, c\neq 0,\,\, b^2= \frac{|c|^2}{a^2}.$$ In that case, the $SU(3)$-structure is LT and the non-vanishing torsion classes are given by $$W_1=\frac{4e^{i\gamma}(a+b^2)}{3ab},\, W_2=\frac{4e^{i\gamma}(2b^2-a)}{3ab}(\omega_{a,b}+\frac{3ib^2}{2}\eta\wedge\bar\eta),$$ where $c=abe^{i\gamma}.$
\end{introthm}
\begin{introcor}\label{cor1}
Using the parameters $a$ and $b$, we can obtain two subclasses of structures summarized in the following table:
\\ \\
\centering
\renewcommand{\arraystretch}{1.6}
\begin{tabular}{lll}
\toprule
\textbf{Torsion classes} &
\textbf{Type of structure} &
\textbf{Values of $a,b,\gamma$} \\
\midrule

$\begin{aligned}
&W_1\neq0,\\
&W_2=W_3=W_4=W_5=0
\end{aligned}$
&
\begin{tabular}[c]{@{}c@{}}
Nearly Kähler\\
$SU(3)$
\end{tabular}
&
$\begin{aligned}
&a>0,\qquad a=2b^2,\\
&\gamma\in\mathbb{R}
\end{aligned}$

\\[1ex]
\midrule
$\begin{aligned}
&W_2\neq0,\\
&W_1=W_3=W_4=W_5=0
\end{aligned}$
&
\begin{tabular}[c]{@{}c@{}}
Symplectic\\
$SU(1,2)$
\end{tabular}
&
$\begin{aligned}
&a=-b^2,\qquad a\neq0,\\
&\gamma\in\mathbb{R}
\end{aligned}$

\\

\bottomrule
\end{tabular}
\end{introcor}
\subsection{Examples} In the regular case, it is known that there are two compact simply connected regular $3$-Sasakian manifolds of dimension $7$ which are the $7$ dimensional sphere and $SU(3)/U(1)$ \cite{Friedrich:1990zg}. With our results, we recover the examples of LT structures on $\C\p^3$ and on the flag manifold $SU(3)/T^2$ known in the literature \cite{Tomasiello_2008,Koerber_2008}. Moreover, we relate our construction to a construction of Eells and Salamon in \cite{EellsSalamon}, where they use the Levi-Civita connection to reverse the direction of the fiber of the standard complex structure on the twistor space of any oriented $4$ dimensional Riemannian manifold. \\
\indent In the quasi-regular case, there are many more examples. Here, we shall be interested in two families of examples. The first one from \cite{Boyer:1998sf} is constructed by a $3$-Sasakian moment map. More precisely, we start with the usual quaternionic vector space $\Hq^{n+1}$ and we consider weighted actions of compact tori $T^{n-1}\subset Sp(n+1)$. For some weights, the $3$-Sasakian quotients will inherit a smooth $3$-Sasakian structure with the property that these quotients are topologically different providing an infinite family of quasi-regular examples where Theorem \ref{theo1} and \ref{theo2} can be applied. One can also construct examples in the other way around by starting with a quaternionic Kähler orbifold and using Theorem $2.3.5$ in \cite{Boyer:1998sf} to get a $3$-Sasakian orbifold of dimension $7$. We apply this result in the case of certain weighted projective spaces obtaining once again a family of topologically different LT-structures with singularities. \\ \\
\indent The paper is organized as follows. In Section \ref{2}, we review the basics of $SU(n)$-structures and the twist construction. In Section \ref{3}, we start with some preliminaries on hyperkähler geometry to define $3$-Sasakian manifolds from the point of view of the cone. We then focus on the $7$-dimensional case in Section \ref{3.2}. Sections \ref{3.4} and \ref{3.5} are devoted to the proofs of our main results, while Section \ref{3.6} discusses some physics considerations. Section \ref{4} presents examples illustrating our construction in the regular,  quasi-regular and orbifold cases\\ \\
\textbf{Acknowledgments.} We would like to thank Eveline Legendre for valuable discussions and helpful comments on an earlier version of this paper.~Q.P.~is partially supported by the grant BRIDGES ANR-FAPESP ANR-21-CE40- 0017. D.T.~is partially supported by the French National Research Agency grant ANR-25-CE57-1445-01.
\section{Preliminaries}\label{2}
\subsection{$SU(n)$-structures}\label{2.1}
We start this section by recalling the basic facts about $SU(n)$-structures. 
\begin{Def}
    An $SU(n)$-structure on a smooth manifold $M$ of dimension $2n$ is an almost hermitian structure $(g,J,\omega)$ on $M$ with a normalized complex volume form $\Omega$, that is a nowhere vanishing section of $\Lambda^{n,0}T^*M$, the space of $(n,0)$-forms on $M$, such that 
    $$\omega\wedge\Omega=0,$$ $$\Omega\wedge\overline\Omega=c_n\frac{\omega^n}{n!},$$ where $c_n=2^n(-i)^{n^2}.$
\end{Def}
Recall that an \textit{almost hermitian structure} $(g,J,\omega)$ on $M$ consists of a Riemannian metric $g$, an almost complex structure $J$, i.e $J\in\Gamma(\End TM), J^2=-1$ and a real $2$-form $\omega$ which is \textit{$J$-compatible}, i.e,  $\omega(J\cdot,J\cdot)=\omega$ and $g=\omega(\cdot,J\cdot).$ \\
\indent An almost complex structure $J$ on $M$ gives rise to a decomposition of complex forms into components of type $(p,q)$ as follows. $J$ splits the complexified tangent bundle $TM^\C=TM\otimes\C$ as $TM^\C=T^{1,0}M\oplus T^{0,1}M$ where $T^{1,0}M=\ker(J-iId)$ is called the holomorphic tangent bundle and $T^{0,1}M=\overline{T^{1,0}M}$ is the anti-holomorphic tangent bundle. Taking the dual of these bundles gives rise to the decomposition of the space of complex $1$-forms into $(1,0)$-forms and $(0,1)$-forms that we respectively denote $\Lambda^{1,0}T^*M$ and $\Lambda^{0,1}T^*M$. Then, a $(p,q)$-form is a section of the bundle $\Lambda^{p,q}T^*M$ defined by $$\Lambda^{p,q}T^*M:=\underbrace{\Lambda^{1,0}T^*M\wedge\cdots\wedge\Lambda^{1,0}T^*M}_{\text{$p$-times}}\wedge\underbrace{\Lambda^{0,1}T^*M\wedge\cdots\wedge\Lambda^{0,1}T^*M}_{\text{$q$-times}}.$$
\indent The action of the exterior derivative $d$ on $(p,q)$-forms is 
$$d(\Lambda^{p,q}T^*M)=\Lambda^{p+2,q-1}T^*M\oplus\Lambda^{p+1,q}T^*M\oplus\Lambda^{p,q+1}T^*M\oplus\Lambda^{p-1,q+2}T^*M.$$ \indent Using this decomposition, one can define the \textit{intrinsic torsion} of an $SU(n)$-structure $(M,g,J,\omega,\Omega)$ 
\begin{Def}\cite{Grayhervella,chiossi2002intrinsictorsionsu3g2}
    Let $(M,g,J,\omega,\Omega)$ be an $SU(n)$-structure on a manifold $M$. The intrinsic torsion of the $SU(n)$-structure is
    $$d\omega=\frac{3}{2^{n-2}}i^{(n+1)^2}\Ima(\overline{W_1}\cdot\Omega)+W_4\wedge\omega+W_3,$$ $$d\Omega=W_1\wedge\omega^2+W_2\wedge\omega+\overline W_5\wedge\Omega,$$ where $W_1$ is a $(n-3,0)$ form, $W_2$ a primitive $(n-2,1)$-form, $W_3$ a real primitive $(2,1)+(1,2)$-form, $W_4$ is a real $(1,0)+(0,1)$-form and $W_5$ is a   $(1,0)$-form.
\end{Def}
    Recall that a $k$-form $\beta$ is \textit{primitive} if $\beta\wedge\omega^{n+1-k}=0$ and that the notation $W_1\cdot$ is the contraction of forms induced by the metric. The $W_i$ are known in the literature as the \textit{torsion classes} and measure the lack of integrability of the $SU(n)$-structures. Indeed, one can see that the almost complex structure is for instance integrable iff $W_1=W_2=0$. The case where all the $W_i$ are zero, called the \textit{torsion free} case, corresponds to a Calabi-Yau structure on $M$.
\subsection{A twist construction}\label{2.2}
In this paper, we are interested in a special type of $SU(n)$-structures, called the LT structures, motivated by their physical relevance for $\mathcal{N}=1$ compactifications for $IIA$ supergravity to $AdS_4$, see \cite{Lust_2005}. 
\begin{Def}\label{LT}
    An LT structure on $M$ is an $SU(n)$-structure with $W_3=W_4=W_5=0$ with $W_1$ and $W_2$ both non-zero.
\end{Def}
In the toric setting in dimension $6$, the construction of such structures was first considered in \cite{Larfors_2010} and then generalized to any arbitrary dimension in \cite{peres2026sunstructuresquotienttorusactions}. The construction is based on the following two definitions:
\begin{Def}\label{chargedef}
    Let $M$ be a smooth manifold with a real torus $T$ acting on it with Lie algebra $\tor$. A form $\alpha$ on $M$ has charge $k\in\tor^*$ if for every $\zeta\in\tor$, $$\lied_{X_\zeta}\alpha=k(\zeta)\alpha,$$ where $X_\zeta$ is the vector field induced by $\zeta$ on $M$.
\end{Def}
\begin{Def}\label{twistdef}
    Let $(M,J)$ be an almost complex manifold with a real torus $T$ acting on it with Lie algebra $\tor$. A $1$-form $\alpha$ on $M$ is a twist form if:
    \begin{enumerate}[i)]
        \item $\alpha$ is nowhere-vanishing and of type $(1,0)$.
        \item $\alpha$ is horizontal, i.e for every $\zeta\in\tor, \alpha(X_\zeta)=0$. 
        \item $\alpha$ has a charge $k_\alpha\in\tor^*$.
    \end{enumerate}
\end{Def}
\begin{theo}\cite{peres2026sunstructuresquotienttorusactions}\label{twisttheorem}
    Let $(M^{2n+2s},g,J,\omega_0,\Omega_0)$ be a Kähler manifold of real dimension $2n+2s$ with an $SU(n+s)$-structure $(\omega_0,\Omega_0)$. Suppose that a compact torus $T^s$ of dimension $s$ acts on $M$ holomorphically and in a hamiltonian way with moment map $\mu$. Consider $\tor:=Lie(T^s)$ the Lie algebra of $T^s$ and choose a basis $e_1,\cdots,e_s$ of $\tor$ and denote $V_i$ the vector field generated by $e_i$ on $M$. Assume that $\Omega_0$ is of charge $q\in\tor^*$ with respect to the $T^s$ action and that there exists a twist form $\alpha$ on $M$ of charge $\frac{q}{2}$. Then, the following forms:
    $$\omega:=\omega_0-\frac{i}{|\alpha|^2}\alpha\wedge\overline{\alpha},$$ $$\Omega:=\overline{\alpha}\wedge(\overline{\alpha}\cdot(\iota_{V_1}\cdots\iota_{V_s}\Omega_0)),$$ are basic with respect to the $T^s$ action and define an $SU(n)$-structure on a regular quotient $N=\mu^{-1}(c)/T^s$ for some $c\in\tor^*$.
\end{theo}
\indent In the statement of the Theorem \ref{twisttheorem}, note that the charge $q$ of $\Omega_0$ takes values in the imaginary numbers as a consequence of the normalized condition of $\Omega_0$. An application of Theorem \ref{twisttheorem} is the construction of LT structures on the complex projective space $\C\p^3$, which first appeared in \cite{Tomasiello_2008,Koerber_2008} and then in \cite{peres2026sunstructuresquotienttorusactions} where all the $SU(3)$-structures coming from this twist theorem in the holomorphic setting were described.
\section{3-Sasakian geometry}\label{3}
\subsection{Hyperkähler manifolds}\label{3.1} The goal of this section is to recall the main facts about $3$-Sasakian manifolds that will be relevant for this paper, as well as  their generalization to orbifolds. $3$-Sasakian manifolds were first considered in \cite{Kuo1970OnAC} using the language of contact structures. In this paper, we shall follow the cone point of view of this type of manifolds presented for example in \cite{Boyer:1998sf} where the hyperkähler geometry is used to describe $3$-Sasakian structures. 
\begin{Def}
    A Riemannian manifold $(M,g)$ is an hyperkähler manifold if there exists three integrable complex structures $J_1,J_2,J_3$ such that $(M,g,J_i)$ is Kähler for $i=1,2,3$ and $J_1,J_2,J_3$ satisfy the quaternionic relations 
    $$J_1^2=J_2^2=J_3^2=J_1J_2J_3=-1.$$
\end{Def}
In other words, a hyperkähler manifold is a Riemannian manifold whose holonomy is contained in the group $Sp(n)$ for some $n$. Consequently, if $(M,g)$ is hyperkähler, then its dimension is $4n$ and has a Calabi-Yau structure since $Sp(n)\subset SU(2n)$. We can explicitly describe the holomorphic volume form of a hyperkähler manifold. \\ 
\indent Let $(M,g,J_i)$ be a hyperkähler manifold and denote $\omega_i:=g(J_i\cdot,\cdot)$ the associated Kähler form induced by $g$ and the complex structure $J_i$. We introduce the following $2$-form $\rho_1:=\omega_2+i\omega_3$. From the definitions, we have the easy fact:
\begin{lemma}
    $\rho_1$ is a non-degenerate, closed, $(2,0)$-form with respect to $J_1$.
\end{lemma}
\begin{proof}
    The non-degeneracy and closure of $\rho_1$ follow from the ones of $\omega_2$ and $\omega_3$. To prove that $\rho_1$ is $(2,0)$ with respect to $J_1$ it suffices to prove that for any vector fields $X,Y$ on $M$, we have
    $$\rho_1(J_1X,Y)=i\rho_1(X,Y).$$ Using the quaternionic relations, we see that $$\omega_2(J_1X,Y)=g(J_2J_1X,Y)=-g(J_3X,Y)=-\omega_3(X,Y),$$ which implies the result. 
\end{proof}
$\rho_1$ is called a \textit{holomorphic symplectic form}. This form provides a way to holomorphically trivialize the canonical line bundle with respect to $J_1$, $K_{M,1}=\Lambda^{2n,0}_{J_1}T^*M$ by considering $\Omega_{(2n)}=\frac{1}{n!}\rho_1^n$. In other words, we have proved that
\begin{prop}
    $(M,g,J_1,\omega_1,\Omega_{(2n)})$ is an $SU(2n)$-structure without torsion.
\end{prop}
\begin{rk}
    Note that we could have done a similar construction using either $J_2$ or $J_3$.
\end{rk}
\subsection{$3$-Sasakian manifolds}\label{3.2} In this paragraph, we describe the main properties of $3$-Sasakian manifolds. 
\begin{Def}
    A Riemannian manifold $(S,g)$ is a $3$-Sasakian manifold if its Riemannian cone $(C(S)=S\times\R_{>0},\bar g=dr^2+r^2g)$ is an hyperkähler manifold.
\end{Def}
The dimension of a $3$-Sasakian manifold is $4n+3$ for some $n$. For a $3$-Sasakian manifold $(S,g)$, we let $J_1,J_2,J_3$ the three complex structures of the hyperkähler structure on its cone and $\omega_1,\omega_2,\omega_3$ the corresponding Kähler forms. In what follows, we shall see $S=S\times\{1\}\hookrightarrow C(S)$ isometrically embedded in its cone. From the cone, we can recover the data of the contact structures, by considering the Reeb vector fields $\xi_i=J_i(r\partial_r)$ which are tangent to $S$, the dual contact $1$-forms $\eta_i=\frac{1}{r^2}\bar g(\xi_i,\cdot)$ and the characteristic $(1,1)$-tensor $\Phi_i=\nabla\xi_i$ defined on $S$ where $\nabla$ is the Levi-Civita connection of $g$ on $S$.
\begin{prop}\cite{Boyer:1998sf}\label{trisasakien} Let $(S,g)$ be a $3$-Sasakian structure with Reeb vector fields $\xi_i$, dual contact $1$-forms $\eta_i$ and characteristic $(1,1)$-tensor $\Phi_i$. The following properties hold:
$$\eta_i(\xi_j)=\delta_{ij},\,\iota_{\xi_i}d\eta_i=0,$$ $$\Phi_i\xi_j=\epsilon_{ijk}\xi_k,$$ $$\Phi_i\Phi_j-\eta_j\otimes\xi_i=-\delta_{ij}+\epsilon_{ijk}\Phi_k,$$ $$[\xi_i,\xi_j]=-2\epsilon_{ijk}\xi_k,$$ $$\omega_i=\frac{1}{2}d(r^2\eta_i).$$

\end{prop}
In what follows, we shall refer to a $3$-Sasakian manifold as $(S,g,\xi_i,\eta_i,\Phi_i)$ to higlight the $3$ compatible Sasakian structures.
\begin{rk}
    It is well-known that the Sasakian manifold is Einstein iff its cone is Calabi-Yau. Since a hyperkähler manifold is in particular Calabi-Yau, every $3$-Sasakian manifold is Einstein.
\end{rk}
\begin{rk}
    Using the three Sasakian structures, we can in fact consider a whole $S^2$ family of Sasakian structures. More precisely, take $u=(a,b,c)\in S^2$ and consider $\xi_u:=a\xi_1+b\xi_2+c\xi_3$. This $\xi_u$ defines another Reeb vector field and consequently another Sasakian structure. Note that this Reeb vector field corresponds to consider the complex structure $J_u:=aJ_1+bJ_2+cJ_3$ on the cone which also defines a Kähler structure associated to the metric $\bar g$. 
\end{rk}
Given a compact $3$-Sasakian manifold $(S,g,\xi_i,\eta_i,\Phi_i)$, there is a natural rank $3$ foliation $\F$ induced by the three Reeb vector fields. Note that this foliation is automatically \textit{quasi-regular} and so the space of leaves $B=S/\F$ is a compact orbifold. If the foliation is further assumed to be \textit{regular}, then $B$ is a genuine manifold and $S$ is a principal $G$-bundle over $B$ for $G=SU(2)$ or $SO(3)$. We can also consider the rank $1$ foliation $\F_u\subset\F$ generated by $\xi_u$ where $u\in S^2$. The foliations $\F_u$ and $\F$ are related by
\begin{prop}\cite{Boyer:1998sf}
    Let $(S,g,\xi_i,\eta_i,\Phi_i)$ be a compact $3$-Sasakian manifold. The following statements are equivalent:
    \begin{enumerate}[(i)]
        \item $\F$ is regular,
        \item For every $u\in S^2,\,\F_u$ is regular,
        \item For one $u\in S^2,\,\F_u$ is regular.
    \end{enumerate}
\end{prop}
We will denote $Z=S/{\F_u}$ for some $u\in S^2$. This $Z$ is called the \textit{twistor space} of $S$ which in general is an orbifold; the name will become clearer in the next paragraph. Since $(S,g)$ is Einstein, the classical result of Sasakian geometry, see for example \cite[Theorem 2.1.2]{Boyer:1998sf}, shows that $Z$ carries a Kähler-Einstein metric $g_Z$ of positive scalar curvature. More precisely, if $\pi_Z : (S,g)\to (Z,g_Z)$ denotes the Riemannian submersion, we have that $g=\pi_Z^*g_Z+\eta_u\otimes\eta_u$ where $\eta_u$ is the dual $1$-form of $\xi_u$. \\ \\
\indent In the sequel, it is convenient to introduce the horizontal and vertical distributions:
$$\cV=\Span\{\xi_1,\xi_2,\xi_3\},\,\,\cH=\cV^\perp.$$ A form $\alpha$ on $S$ is $\textit{horizontal}$ if $\iota_{\xi_i}\alpha=0$ or equivalently, defines a section of the exterior algebra of the dual horizontal bundle. Pointwise, $\cH$ can be identified with $TB$ and so any section of $\Lambda\cH^*$ can locally be seen as a form on $B$ and globally if they are $\textit{basic}$ that is $\iota_{\xi_i}\alpha=\lied_{\xi_i}\alpha=0$. $B$ is also equipped with a Riemannian metric $g_B$ \cite{Boyer:1998sf} which is characterized by $$g|_{\cH}=\pi_B^*g_B$$ where $\pi_B : S\to B$. $g$ and $g_B$ are related by $$g=\pi^*_Bg_B+\sum_{i=1}^3 \eta_i\otimes\eta_i.$$ 
\indent We end this section by reviewing a quotient reduction of $3$-Sasakian manifolds using moment maps that we are going to use in Section \ref{4.2}. We first define an isometry of a $3$-Sasakian structure.
\begin{Def}
    Let $(S,g,\xi_i)$ be a $3$-Sasakian manifold. A $3$-Sasakian isometry of $(S,g,\xi_i)$ is a diffeomorphism $\phi$ of $S$ such that $$\phi^*g=g,\, \phi_*\xi_i=\xi_i.$$ We denote $\Isom(S,g,\xi)$ the group of $3$-Sasakian isometries of $(S,g,\xi_i)$. 
\end{Def}
Note that any element $\phi\in\Isom(S,g,\xi)$ can be lifted to an isometry on the hyperkähler cone and also preserves the hyperkähler structure. The action of a subgroup $G$ acting by isometries on $(C(S),\bar g)$ and preserving the complex structures gives rise to a hyperkähler moment map $\mu : C(S)\to\g^*\otimes\R^3$, see \cite{Hyperkahlermetricssupersymmetry} for more details. The restriction of this map to $S$ will be called a \textit{$3$-Sasakian moment map}. If $G$ is a compact connected Lie group acting by $3$-Sasakian isometries, then Proposition $6.1.2$ in \cite{Boyer:1998sf} shows that there is a unique $3$-Sasakian moment map $\mu_S$ which will be relevant in the following:
\begin{theo}\label{reductiontheorem}\cite[Theorem 6.1.5]{Boyer:1998sf} Let $(S,g,\xi_i)$ be a $3$-Sasakian manifold of dimension $4n+3$. Assume that there is a connected compact Lie group $G$ acting smoothly and properly on $S$ by $3$-Sasakian isometries. Let $\mu_S$ be the corresponding $3$-Sasakian moment map and suppose that $0$ is a regular value of $\mu_S$ and that $G$ acts freely on $\mu_S^{-1}(0)$. Then, the $3$-Sasakian quotient $S/\!\!/G=\mu^{-1}(0)/G$ is of dimension $4(n-\dim G)+3$ and admits a $3$-Sasakian structure $(\hat g,\hat\xi_i)$ which are characterized by
$$g|_{\mu^{-1}_S(0)}=\pi^*\hat g,$$ $$\pi_*(\xi_i|_{\mu^{-1}_S(0)})=\hat\xi_i,$$ where $\pi : \mu^{-1}_S(0)\to S/\!\!/G $. 
\end{theo}
We end this section by recalling that a \textit{3-Sasakian orbifold} is a Riemannian orbifold that admits three characteristic vector fields satisfying the relations of Proposition \ref{trisasakien} and are preserved by the local unirformizing groups, cf.~\cite[Remark 1.2.3]{Boyer:1998sf}. In particular, all the forms and tensors we have considered for the manifold case are also globally well-defined for orbifolds and everything that have been discussed so far for $3$-Sasakian manifolds also work for $3$-Sasakian orbifolds. 
We shall call the foliation generated by a Reeb vector  $\xi_u$ \textit{quasi-regular} when $\xi_u$ integrates to an almost-free, effective $S^1$-action, and \textit{regular} when this action is free.

\subsection{$3$-Sasakian manifolds of dimension 7}\label{3.3} In this paragraph, we focus on the $7$ dimensional case as it will be the starting point of our main results. Let $(S,g,\xi_i,\eta_i,\Phi_i)$ be a $3$-Sasakian manifold of dimension $7$. Then $S$ fibers over $Z$ with $S^1$ fibers which in turn fibers over $B$ with $S^2\simeq\C\p^1$ fibers. We recall here how $Z$ appears as the twistor space of $B$ following \cite{quaternionkahler}. \\ 
\indent In the regular case, $B$ is a quaternionic Kähler manifold, there exists a rank $3$-subbundle $E\subset\End TB$ which is locally spanned by three almost complex structures $I,J,K$, stable by the Levi-Civita connection and satisfying the quaternionic relations. This gives an identification with the space of self-dual bivectors $\Lambda^+TB$. Indeed, take $(e_1,\cdots,e_4)$ a local orthonormal frame, then $\Lambda^+TB$ is spanned by $e_1\wedge e_2+e_3\wedge e_4, e_1\wedge e_3-e_2\wedge e_4, e_1\wedge e_4+e_2\wedge e_3$ and by interpreting each one of these three forms as a local Kähler form gives the local $I,J,K$ and allows to identify $\Lambda^+TB$ with $E$. The twistor space of $B$ is then defined as the sphere bundle of $\Lambda^+TB$ which then is identified with the sphere of $E$ in $\End TB$. In the $3$-Sasakian case, we can describe $E$ as follows. For any $i$, $\Phi_i$ can be restricted to $\mathcal{H}$ and thus defines locally three almost complex structures which satisfy the quaternionic relations as the $\Phi_i$ satisfy these identities on $\mathcal{H}$ by Proposition \ref{trisasakien}. Hence, $E\simeq S\underset{G}{\times}\R^3$ where $G$ is either $SO(3)$ or $SU(2)$. Choose $u\in S^2$, the group $G$ acts transitively on $S^2$ and if $G_u$ denotes the stabilizer of $u$, which corresponds to the leaf of the foliation $\F_u$, then we have the isomorphisms
$$G_u\simeq S^1, \, \, G/G_u\simeq S^2.$$ The twistor space of $B$ then reads as 
$$\mathcal{S}(E)\simeq S\underset{G}{\times}S^2\simeq S\underset{G}{\times}(G/G_u)\simeq S/G_u\simeq Z.$$
\indent Using the classification of compact Kählerian twistor spaces obtained in \cite{Hitchintwistor, twistorspaces}, Friedrich and Kath proved the following result:
\begin{theo}\label{theoremclassification}\cite[Theorem 7]{Friedrich:1990zg}
    Let $(S,g)$ be a compact simply connected regular $7$-dimensional $3$-Sasakian manifold. Then, $S$ is isometric to either the $7$-sphere or to $SU(3)/U(1)$. 
\end{theo}
The twistor spaces associated to these two $3$-Sasakian manifolds are respectively $\C\p^3$ and the flag manifold $F=SU(3)/T^2.$
\begin{rk}
    This construction remains valid in the case where both $S$ and the corresponding twistor space $Z$ may have an orbifold structure.
\end{rk}
\subsection{Proof of Theorem \ref{theo1}}\label{3.4} In this section, we give the proof of Theorem \ref{theo1}. Note that the arguments for the regular case that we will give also work for the quasi-regular case. Indeed, it is known, see \cite{caramello2022introductionorbifolds}, that if $S$ is  an orbifold with an almost free, effective action by a compact
Lie group $G$, then  the quotient space 
$S/G$ has a natural orbifold structure, and,  moreover, forms on $S$ that are basic with respect to  $G$ descend to well-defined forms on $S/G$. Therefore, if a form is basic on $S$ with respect to the $S^1$ action generated by $\xi_3$, then it will induce a well-defined orbifold form on $Z$. \\
\indent Let $(S,g,\xi_i,\eta_i,\Phi_i)$ be a regular $3$-Sasakian manifold of dimension $7$. Consider its cone $C(S)$ with the radius function $r : C(S)\to\R_{>0}$. We fix the Sasakian structure induced by $(\xi_3,\eta_3,\Phi_3)$. We now view $(C(S),J_3,\omega_3)$ as a Calabi-Yau $4$-fold with holomorphic volume form $\Omega_{(4)}:=\frac{1}{2}(\omega_1+i\omega_2)^2$. It is well-known that $\mu:=\frac{1}{2}r^2$ is a moment map for the hamiltonian action of the locally free action generated by $\xi_3$ on $C(S)$. The following lemma will be useful to compute the different charges:
\begin{lemma}
    We have
    $$\lied_{\xi_3}\eta_i=-2\epsilon_{ij3}\eta_j.$$
\end{lemma} 
\begin{proof}
    This follows from the relation $[\xi_i,\xi_j]=-2\epsilon_{ijk}\xi_k$ and the definition of the contact forms. 
\end{proof}
Using Proposition \ref{trisasakien}, we see that $\Omega_{(4)}$ has a charge $q=4i$. Moreover, this proves
\begin{prop}
    The form $\eta=r^2(\eta_1+i\eta_2)$ is a twist form of charge $2i$ on $C(S)$ for the action of $\xi_3$.
\end{prop}
\begin{proof}
    By Proposition \ref{trisasakien}, $\eta$ is horizontal for $\xi_3$ and nowhere-vanishing as $\eta(\xi_1)=r^2$. Note that $J_3\eta_1=-\frac{1}{r^2}\overline{g}(\xi_1,J_3\cdot)=\frac{1}{r^2}g(\xi_2,\cdot)=\eta_2$ and so $\eta$ is $(1,0)$. Since $r^2$ is $S^1$-invariant, we get that $\eta$ is indeed of charge $2i$.
\end{proof}
Theorem \ref{twisttheorem} then provides an $SU(3)$-structure on the quotient $Z=S/\mathcal{F}_3$. Our next purpose is to compute the torsion of this $SU(3)$-structure. To do this, let $$\omega:=\omega_3-\frac{i}{|\eta|^2}\eta\wedge\bar\eta,\,\, ~~~\Omega:=-i\bar\eta\wedge(\bar\eta\cdot(\iota_{\xi_3}\Omega_{(4)}))$$ be the $SU(3)$-structure on $Z$.
First, remark that $|\eta|^2=r^2$ and so is constant on a level set of the action. We establish a relation on the differential of $\eta$. We have
$$d\eta=d(r^2\eta_1)+id(r^2\eta_2)=2(\omega_1+i\omega_2).$$ It is more convenient to write this differential in terms of $\Omega_{(4)}$. For that, remark that $$(d\eta)^2=8\Omega_{(4)}.$$ Contracting successively by $\xi_3$ and by $\bar\eta$ gives 
$$r^2d\eta+2r^2\eta\wedge(d\log r+i\eta_3) =-2i\bar\eta\cdot\iota_{\xi_3}\Omega_{(4)}.$$ With this relation, we can compute $d\omega$ on $S$, i.e on the level set $\{r=1\}$. 
\begin{align*}
    d\omega&=2\Ima(d\eta\wedge\bar\eta)\\
    &=2\Ima(-2i\bar\eta\cdot\iota_{\xi_3}\Omega_{(4)}\wedge\bar\eta+2i\eta\wedge\bar\eta\wedge\eta_3)\\
    &=\Ima(4\Omega)
\end{align*}
We deduce that $W_1=\frac{8}{3},W_3=W_4=0$. For $d\Omega$ one gets
\begin{align*}
    d\Omega=&-i\bigg(d\bar\eta\wedge(\bar\eta\cdot\iota_{\xi_3}\Omega_{(4)})-\bar\eta\wedge d(\bar\eta\cdot\iota_{\xi_3}\Omega_{(4)})\bigg)\\
    &=-i\bigg(2i(\eta\cdot\iota_{\xi_3}\overline{\Omega_{(4)}}+\bar\eta\wedge\eta_3)\wedge(\bar\eta\cdot\iota_{\xi_3}\Omega_{(4)})+\bar\eta\wedge(d\eta\wedge\eta_3-\eta\wedge d\eta_3)\bigg)\\
    &=4(\omega+\frac{i}{2}\eta\wedge\bar\eta)^2-i\eta\wedge\bar\eta\wedge d\eta_3\\
    &=4(\omega^2+i\eta\wedge\bar\eta\wedge\omega)-2i\eta\wedge\bar\eta\wedge\omega\\
    &=4\omega^2+2i\eta\wedge\bar\eta\wedge\omega\\
    &=W_1\omega^2+\frac{W_1}{2}(\omega+\frac{3i}{2}\eta\wedge\bar\eta)\wedge\omega.
\end{align*}
We conclude that $W_2=\frac{W_1}{2}(\omega+\frac{3i}{2}\eta\wedge\bar\eta)$ (one can easily check that $W_2$ is indeed a primitive form) and $W_5=0$. So the structure obtained is LT which proves Theorem \ref{theo1}. \\ 
\subsection{Proof of Theorem \ref{theo2}}\label{3.5}
Here, we give the proof of Theorem \ref{theo2} which is a generalization of Theorem \ref{theo1} and consider a family of $SU(3)$-structures on the twistor space $Z$. As for Theorem \ref{theo1}, it is sufficent to prove it when $S$ is a manifold. Indeed, the proof will rely on showing that the forms we consider descend to the quotient, which means that they are basic with respect to the locally free $S^1$ action generated by $\xi_3$, see \cite[\S3.5]{caramello2022introductionorbifolds}.\\ \indent   We will directly work on $S$ and more precisely on $\cH$ to get an $SU(3)$-structure on the quotient. For that, we introduce the $2$-forms $\alpha_i$ on $S$ by $$2\alpha_i:=d\eta_i-\epsilon_{ijk}\eta_j\wedge\eta_k.$$ These $\alpha_i$ are horizontal and we shall see them as sections of $\Lambda^2\cH^*$. Note that for every $X,Y$ sections of $\cH$ $$\alpha_i(X,Y)=g(\Phi_iX,Y).$$ Consequently, if $\mu_B$ is the Riemannian volume form on $(B,g_B)$ that we pullback to $S$, then  we infer that $$\mu_S=\mu_B\wedge\eta_1\wedge\eta_2\wedge\eta_3, \,\alpha_i\wedge\alpha_j=2\delta_{ij}\mu_B,$$ where $\mu_S$ is the Riemannian volume form of $(S,g)$. \\ \\
For $a,b\in\R$ and $c\in\C$, we define $$\omega_{a,b}:=a\alpha_3-\frac{i}{2}b^2\eta\wedge\bar\eta,~ \Omega_c:=c\alpha\wedge\bar\eta$$ with $\alpha:=\alpha_1+i\alpha_2$. Note that by Proposition \ref{trisasakien}, $\alpha_3, \eta\wedge\bar\eta$ and $\alpha\wedge\bar\eta$ are basic. They therefore descend to orbifold forms on $Z$ \cite[\S3.5]{caramello2022introductionorbifolds}, which we denote by the same symbols. 

\begin{lemma}
    For every $a,b\in\R, c\in\C$ $$\omega_{a,b}\wedge\Omega_c=0.$$
\end{lemma}
\begin{proof}
    Easy consequence from $\alpha_3\wedge\alpha=0$ and $\bar\eta\wedge\bar\eta=0$.
\end{proof}
\begin{prop}
    For every $a,b\in\R, c\in\C$ $$i\Omega_c\wedge\overline{\Omega_c}=-8|c|^2\mu_Z,$$ $$\omega_{a,b}^3=-6a^2b^2\mu_Z,$$ with $\mu_Z$ the Riemannian volume form of $g_Z$ on $Z$ pulled back on $S$. 
\end{prop}
\begin{proof}
First, remark that with these identifications, $\mu_Z=\mu_B\wedge\eta_1\wedge\eta_2$. 
    \begin{align*}
        \Omega_c\wedge\overline{\Omega_c}&=-|c|^2\alpha\wedge\bar\alpha\wedge\eta\wedge\bar\eta\\
        &=2i|c|^2(\alpha_1^2+\alpha_2^2)\wedge\eta_1\wedge\eta_2\\
        &=8i|c|^2\mu_B\wedge\eta_1\wedge\eta_2\\
        &=8i|c|^2\mu_Z\\
        \omega_{a,b}^3&=(a\alpha_3-\frac{i}{2}b^2\eta\wedge\bar\eta)^3\\
        &=-\frac{3a^2b^2i}{2}\alpha_3^2\wedge\eta\wedge\bar\eta\\
        &=-\frac{3a^2b^2i}{2}2\mu_B\wedge(-2i\eta_1\wedge\eta_2)\\
        &=-6a^2b^2\mu_Z
    \end{align*}
\end{proof}
 We deduce that $(\omega_{a,b},\Omega_c)$ defines an $SU(3)$-structure on $Z$ iff $c\neq 0, a\neq 0$ and $b^2=\frac{|c|^2}{a^2}$. Note that the metric induced by $(\omega_{a,b},\Omega_c)$ may not be positive definite. In the following, we describe the parameters $a,b,c$ for which the corresponding metric is positive definite.\\ \indent Let $(e_1,\cdots,e_4)$ be the pullback on $S$ of a local oriented orthonormal coframe on $B$. Then, $(e_1,\cdots,e_4,\eta_1,\eta_2,\eta_3)$ is a local oriented orthonormal coframe on $S$ and so $g$ can be written $$g=e_1^2+\cdots+e_4^2+\eta_1^2+\eta_2^2+\eta_3^2.$$ We can also write $$\mu_B=e_1\wedge e_2\wedge e_3\wedge e_4,$$ $$\alpha_1=e_1\wedge e_2+e_3\wedge e_4,$$ $$\alpha_2=e_1\wedge e_3-e_2\wedge e_4,$$ $$\alpha_3=e_1\wedge e_4+e_2\wedge e_3.$$ Consider $f_1:=e_1+ie_4, f_2:=e_2+ie_3$ so that $\alpha=f_1\wedge f_2$ and so locally, the induced almost complex structure is defined in a way that $f_1,f_2$ and $\bar\eta$ are $(1,0)$-forms. Thus, we get that the metric on $Z$ induced by $\omega_{a,b}$ and $\Omega_c$ is
$$g_{a,b,c}=\frac{a}{2}(f_1\otimes\bar f_1+\bar f_1\otimes f_1+f_2\otimes\bar f_2+\bar f_2\otimes f_2)+b^2(\eta_1\otimes\eta_1+\eta_2\otimes\eta_2).$$ Therefore, $g_{a,b,c}$ is positive definite iff $a>0$. 
\begin{rk}
    In the case where $a<0$, the metric has signature $(2,4)$ and so we get an $SU(1,2)$-structure rather than an $SU(3)$-structure. 
\end{rk}
We introduce the usual moduli of LT structures \cite{Larfors_2010,Larfors_2013}, namely we write $c=abe^{i\gamma}$ for some constant phase $\gamma\in\R$ which can be absorbed in the definiton of $\Omega_c$ that we rename $\Omega_{a,b,\gamma}$. To compute the torsion of ($\omega_{a,b}$,$\Omega_{a,b,\gamma}$), we need the following identities:
\begin{lemma}
    On $S$, the following relations hold:
    $$d\eta=2\alpha-2i\eta\wedge\eta_3$$ and $$d\alpha=2i\alpha\wedge\eta_3-2i\alpha_3\wedge\eta.$$
\end{lemma}
\begin{proof}
Using the definitions of the $\eta_i$ and $\alpha_j$, one gets  
    \begin{align*}
        d\eta&=d(\eta_1+i\eta_2)\\
        &=2\alpha_1+2\eta_2\wedge\eta_3+i(2\alpha_2-2\eta_1\wedge\eta_3)\\
        &=2\alpha-2i\eta\wedge\eta_3\\
        d\alpha&=id(\eta\wedge\eta_3)\\
        &=i(d\eta\wedge\eta_3-\eta\wedge d\eta_3)\\
        &=i(2\alpha\wedge\eta_3-2\alpha_3\wedge\eta)
    \end{align*}
\end{proof}
Similarly as for Theorem \ref{theo1}, we compute $d\omega_{a,b}$ and $d\Omega_{a,b,\gamma}$ to get
$$d\omega_{a,b}=2(a+b^2)\Ima\bigg(\frac{e^{-i\gamma}}{ab}\Omega_{a,b,\gamma}\bigg)$$ and $$d\Omega_{a,b,\gamma}=c(8\mu_B-2i\alpha_3\wedge\eta\wedge\bar\eta).$$ These expressions imply that $$W_1=\frac{4e^{i\gamma}(a+b^2)}{3ab},$$ $$W_2=\frac{4e^{i\gamma}(2b^2-a)}{3ab}\bigg(\omega_{a,b}+\frac{3ib^2}{2}\eta\wedge\bar\eta\bigg)$$ and $$W_3=W_4=W_5=0.$$ \indent From the expressions of $W_1$ and $W_2$, we directly get Corollary \ref{cor1}.
\begin{rk}
    The almost complex structure induced by $\Omega_{a,b,\gamma}$ is never integrable as $W_1$ and $W_2$ cannot be both zero simultaneously. 
\end{rk}
\begin{rk}
    Theorem \ref{theo1} is a particular case of Theorem \ref{theo2} where we have chosen $a=b=1$ and $\gamma=0$. 
\end{rk}
\subsection{Orientifold sources}\label{3.6}

Let us set $\gamma=\frac{\pi}{2}$ for simplicity and use Theorem \ref{theo2} to obtain
\[
d\omega_{a,b}= -\frac{2(a+b^2)}{ab}~\!\text{Re}(\Omega_{a,b})~;~~
d W_2=\frac{8i(2b^2-a)^2}{3a^2b^2}\text{Re}(\Omega_{a,b})
\]
Moreover, for an LT structure that in addition obeys $d W_2\propto \text{Re}(\Omega_{a,b})$, 
it can be shown that \cite{Lust_2005}
\[
d W_2=\frac{i}{8}|W_2|^2\text{Re}(\Omega_{a,b})
~,\]
where $|W_2|^2:=(W_2)^{mn}(\overline{W}_2)_{mn}$. 
An AdS$_4$ solution without orientifold O6-planes must obey the constraint \cite{Koerber_2008} 
\[
3|W_1|^2-|W_2|^2\geq 0~.
\]
Violation of this condition signals the presence of (smeared) orientifolds, with \cite{DeWolfe_2005} being the prototypical example. In string theory, orientifolds arise by quotienting by a symmetry that combines an involution of the internal manifold with reversal of the string orientation. In the supergravity approximation, the fixed-point loci of the involution act as localized sources, which in the smeared case are replaced by smooth distributions. 
Rewriting the AdS$_4$ constraint in terms of the moduli, we eventually get that 
\[
a\geq b^2\geq\frac15 a>0
~.\]

\section{Examples}\label{4}
\subsection{Regular examples}\label{4.1}
The recipe explained in the last section gives a systematic method to endow any $S^1$ quotient of a $3$-Sasakian manifold with an LT structure. Using the classification of compact regular $3$-Sasakian manifolds given by Theorem \ref{theoremclassification}, we recover the examples of $\C\p^3$ and the flag manifold $SU(3)/T^2$.\\ 
\indent In \cite{EellsSalamon}, the authors constructed two almost complex structures $J_1$ and $J_2$ on the twistor space of any oriented $4$-dimensional Riemannian manifold. We shall relate this construction to ours. Let $(S,g,\xi_i,\eta_i,\Phi_i)$ be a regular $3$-Sasakian manifold with twistor space $Z$ and quaternionic Kähler base $B$. Consider the quotient maps $\pi_1 : S\to Z$, $\pi_2 : Z\to B$ and $\pi:=\pi_2\pi_1 : S\to B$. We introduce the contact distribution $\cD=\ker\eta_3$ and recall that $\cH=\ker\eta_1\cap\ker\eta_2\cap\ker\eta_3$. In particular, we have the following identification: $$\cD=\cH\oplus\R\xi_1\oplus\R\xi_2.$$ \indent Moreover, we know that the differential of $\pi_1$ gives an isomorphism between $TZ$ and $\cD$. Similarly, the map $\pi_2$ provides a splitting of $TZ=T^V\oplus T^H$ into its vertical and horizontal components. Putting all of these identifications together, we get that pointwise $$T^H\simeq\cH\simeq TB,\qquad T^V\simeq\R\xi_1\oplus\R\xi_2.$$\indent  On $\cD$, the tensor $\Phi_3$ defines a CR structure which induces an almost complex structure on $Z$ which corresponds to the integrable complex structure $J_1$ in \cite{EellsSalamon}. This $J_1$ corresponds to requiring $\xi_1-i\xi_2$ to be $(1,0)$. Eells and Salamon also consider an almost complex structure $J_2$ which acts as $J_1$ on $T^H$ and with an opposite sign on $T^V$. This implies that the form $\bar\eta$ is $(1,0)$ for this complex structure $J_2$. Moreover, the local forms $f_1$ and $f_2$ are also $(1,0)$ with respect to $J_2$. Thus the almost complex structure induced by our construction coincides with the Eells–Salamon structure $J_2$.
\subsection{Quasi-regular and orbifold examples}\label{4.2} In this section, we focus on the quasi-regular and orbifold cases. 

Recall that any compact $3$-Sasakian manifold is automatically quasi-regular and it implies that the twistor space $Z$ is in general an orbifold rather than a manifold. We give here two families of examples. \\
\indent The first family consists of quasi-regular toric $3$-Sasakian manifolds using torus actions. We recall here the construction. Consider the sphere $S^{4n+3}$ in the quaternionic space $\Hq^{n+1}.$ Consider an embedded torus $T^k$ with $k\leq n+1$ in $Sp(n+1)$ and an integral matrix $P=(a^i_j)\in\mathcal{M}_{k,n+1}(\Z)$, $T^k$ then acts on $\Hq^{n+1}$ by  $$(\phi_1,\cdots,\phi_k)\in T^k\mapsto A_P:=\begin{pmatrix}
    \prod_{i=1}^k\phi_i^{a^i_1}&\cdots&0\\
    \vdots&\ddots&\vdots\\
    0&\cdots&\prod_{i=1}^k\phi_i^{a^i_{n+1}}
\end{pmatrix}\in Sp(n+1).$$ Under a certain condition called \textit{admissibility} satisfied by the matrix $P$ in \cite[Definition 7.1.7]{Boyer:1998sf}, Theorem \ref{reductiontheorem} implies the following result:
\begin{prop}\cite[Theorem 6.3.1, 7.1.2]{Boyer:1998sf}
    Let $P$ be an admissible $k\times (n+1)$ integral matrix and denote $T^k_P$ the action of $A_P$ on $\Hq^{n+1}$. Then, the quotient $S(P)=S^{4n+3}/\!\!/T^k_P$ is a smooth $3$-Sasakian manifold.
 \end{prop}
 In particular, for the case $k=n-1$, it produces a family of $7$ dimensional $3$-Sasakian manifolds provided $P$ is well-chosen. One can prove that for any $k>1$, there exists an admissible $P\in\mathcal{M}_{k,k+2}(\Z)$ that gives a genuine $3$-Sasakian manifold of dimension $7$, which we denote by $S(k)$. The advantage of considering this family is justified by the following proposition: 
 \begin{prop}\cite[Proposition 7.5.11]{Boyer:1998sf} Let $k>1$ be an integer. Then, the second Betti number $b_2$ of $S(k)$ is equal to $k$. 
\end{prop}
This family shows that unlike the regular case, the quasi-regular $3$-Sasakian geometry is much richer and offers a lot of possible examples where our main Theorems \ref{theo1} and \ref{theo2} can be applied. In particular, Corollary \ref{cor1} shows that the associated twistor spaces can be endowed with nearly Kähler structures. \\ 
\indent Another possibility to construct examples is by 
\begin{theo}\cite[Theorem 2.3.5]{Boyer:1998sf} Let $B$ be a quaternionic Kähler orbifold with positive scalar curvature. Then, there is a principal $SO(3)$ orbibundle $S$ over $B$ which admits an orbifold $3$-Sasakian structure.
\end{theo}
Knowing that the quaternionic Kähler property is equivalent to the positive self-dual Einstein property in dimension $4$, Galicki and Lawson proved that the weighted projective spaces carry such metrics in \cite{Galicki1988}. More precisely, let $p,q,r$ be three coprime positive integers and consider the $S^1$ action on $\C^3$ given by $$\lambda\cdot(z_1,z_2,z_3):=(\lambda^pz_1,\lambda^qz_2,\lambda^rz_3).$$ The quotient $\C\p^2(p,q,r)$ has an orbifold structure and satisfies
\begin{theo}\label{weigthedprojectivespace}\cite[Theorem 4.26]{Galicki1988} Let $q\leq p$ be positive coprime integers with $p+q$ odd. Then, $\C\p^2(2p,p+q,p+q)$ carries a positive self-dual Einstein metric.
\end{theo}
Consequently, the twistor orbifold of every weighted projective space  of Theorem \ref{weigthedprojectivespace} carries the LT-structure family of Theorem \ref{theo2}.
\section{Conclusion and perspectives.} Our new approach to construct LT-structures using $3$-Sasakian geometry allowed us to describe examples of $6$-dimensional manifolds admitting LT-structures that have been known in the literature, by applying the twist theorem \cite{peres2026sunstructuresquotienttorusactions}. The proof also works in the quasi-regular and orbifold setting, which gives numerous new examples of orbifolds admitting LT-structures. In that perspective, the following questions could be adressed: 
\begin{quest}\label{quest1}
    Can we smooth out the singularities of our $SU(3)$-structures in the quasi-regular and orbifold case? If so, can we preserve some of the torsion classes? 
\end{quest}
 It is also important to note that our results can only be applied for $3$-Saskian manifolds and orbifolds of dimension $7$, because of the requirement of the charge of the twist form. Even though only this case seems relevant for string theory, it would be interesting to know if LT-structures on smooth and non compact manifolds exist in higher dimensions. As far as we know, there are no examples of LT-structures in dimension greater than $6$ that are not direct products of lower dimensional LT-structure manifolds.
 \begin{quest}\label{quest2}
     Can we adopt this framework in higher dimension?
 \end{quest}
A possibility that we have not investigated is using Theorem $C$ of \cite{peres2026sunstructuresquotienttorusactions} where more than one twist form is used. Question \ref{quest1} and Question \ref{quest2} are related as we know examples of spaces admitting such structures with singularities in higher dimension.

\printbibliography[title=References]
\end{document}